\documentclass[12pt]{amsart}
\usepackage{amsmath,amsthm,latexsym,amscd,amsbsy,amssymb,amsfonts,fleqn,leqno,
euscript, pb-diagram}

\newtheorem{theorem}{Theorem}[section]
\newtheorem{lemma}[theorem]{Lemma}
\newtheorem{proposition}[theorem]{Proposition}
\newtheorem{cor}[theorem]{Corollary}

\theoremstyle{definition}
\newtheorem{definition}[theorem]{Definition}

\newtheorem{question}[theorem]{Question}

\theoremstyle{remark}

\numberwithin{equation}{section}

\begin{document}

\title[Functional extenders and set-valued retractions]
{Functional extenders and set-valued retractions}

\author[Robert Alkins]{Robert Alkins}
\address{North Bay, ON, Canada}
\email{alkins.rob@gmail.com}

\author[V. Valov]{Vesko  Valov}
\address{Department of Computer Science and Mathematics, Nipissing University,
100 College Drive, P.O. Box 5002, North Bay, ON, P1B 8L7, Canada}
\email{veskov@nipissingu.ca}
\thanks{The author was partially supported by NSERC Grant 261914-08.}



\date{}

%%%%%%%%%% End topmatter %%%%%%%%%%%%%%%%%%%%%

\begin{abstract}
We describe the supports of a class of real-valued maps on $C^*(X)$
introduced by Radul \cite{r0}. Using this description, a
characterization of compact-valued retracts of a given space in
terms of functional extenders is obtained. For example, if $X\subset
Y$, then there exists a continuous compact-valued retraction from
$Y$ onto $X$ if and only if there exists a normed weakly additive
extender $u\colon C(^*(X)\to C^*(Y)$ with compact supports
preserving $\min$ (resp., $\max$) and weakly preserving $\max$
(resp., $\min$). Similar characterizations are obtained for upper
(resp., lower) semi-continuous compact-valued retractions. These
results provide characterizations of (not necessarily compact)
absolute extensors for zero-dimensional spaces, as well as absolute
extensors for one-dimensional spaces, involving non-linear
functional extenders.  
\end{abstract}

\maketitle\markboth{}{Extenders and set-valued retractions}

%%%%%%%%%%%%%%%%%%%%%%%%%%%%%%%%%%%%%%%%%%%%%%%%%%%%%%%%%%%%%%%%
%%%%%%%%%%%%%%%%%%%%%%%%%%%%%%%%%%%%%%%%%%%%%%%%%%%%%%%%%%%%%%%%

%%%%%%%%%%%%%%%%%%%%%%%%%%%%%%%%%%%%%%%%%%%%%%%%%%%%%%%%%%%%%%%%%%%%%%

\section{Introduction}
All spaces in the paper are assumed to be Tychonoff. Continuous (and
bounded) real-valued functions on $X$ are denoted, respectively, by
$C(X)$ and $C^*(X)$.

Some purely topological properties have been characterized using
functional extenders. For example, {\em Dugundji spaces} were
defined by Pelczynski \cite{p} in the terms of linear extension
operators between function spaces. Later, Haydon \cite{ha} proved
that a compactum $X$ is a Dugundji space iff it is an absolute
extensor for 0-dimensional spaces, notation $AE(0)$-spaces. Another
results of this type are Shapiro's characterization \cite{sh} of
compact absolute extensors for one-dimensional spaces (br., $AE(1)$)
in terms of extenders between non-negative function spaces and the
second authors's characterization \cite{v3} of (not necessarily
compact) absolute extensor for 0-dimensional spaces. Following this
line, the second author \cite{v1} obtained recently a
characterization of {\em $\kappa$-metrizable} compacta involving
special function extenders.

In this paper we provide another result in this direction by
characterizing set-valued retracts of a given space in terms of
functional extenders. Recall that a map $u\colon C^*(X)\to C^*(Y)$,
where $X$ is a subspace of $Y$, is called an {\em extender} if
$u(f)$ extends $f$ for all $f\in C^*(X)$. Every map $u\colon
C^*(X)\to C^*(Y)$ generates the maps (called {\em functionals})
$\mu_y\colon C^*(X)\to\mathbb R$, $y\in Y$, defined by
$\mu_y(f)=u(f)(y)$. We consider functionals which are {\em normed},
{\em weakly additive}, {\em preserving $\max$ or $\min$} and {\em
weakly preserve $\min$ or $\max$}. This class of functionals was
introduced by Radul \cite{r0}: A functional $\mu\colon
C^*(X)\to\mathbb R$ is said to be (i) normed, (ii) weakly additive,
(iii) preserving $\max$, and (iv) weakly preserving $\min$, if for
every $f,g\in C^*(X)$ and every constant function $c_X$ we have: (i)
$\mu(1_X)=1$, (ii) $\mu(f+c_X)=\mu(f)+c$, (iii)
$\mu(\max\{f,g\})=\max\{\mu(f),\mu(g)\}$, (iv)
$\mu(\min\{f,c_X\})=\min\{\mu(f),c\}$.  We say that $\mu$ {\em
preserves $\min$} provided $\mu$ satisfies equality (iii) with
$\max$ replaced by $\min$. Similarly, $\mu$ {\em weakly preserves
$\max$} if $\mu$ satisfies condition (iv) with $\min$ replaced by
$\max$.

A map $u:C^*(X)\rightarrow C^*(Y)$ is normed, weakly additive,
preserves $\max$ and weakly preserves $\min$ (resp., preserves
$\min$ and weakly preserves $\max$) provided $u$ satisfies the
corresponding equalities above with the constants $c$ replaced by
the constant functions $c_Y$ on $Y$. Obviously, $u$ has each of
these properties if and only if all functionals $\mu_y$, $y\in Y$,
have the same property.

The set of all normed, weakly additive functionals on $C^*(X)$ which
preserve $\max$ (resp. $\min$) and weakly preserve $\min$ (resp.,
$\max$) is denoted by $\mathfrak{R}_{max}(X)$ (resp.,
$\mathfrak{R}_{min}(X)$). The topology of these two spaces is
inherited from the product $\mathbb R^{C^*(X)}$. We describe the
supports of  the functionals from
$\mathfrak{R}_{max}(X)\cup\mathfrak{R}_{min}(X)$ and introduce the
subspaces $\mathfrak{R}_{max}(X)_{c}\subset\mathfrak{R}_{max}(X)$
and $\mathfrak{R}_{min}(X)_c\subset\mathfrak{R}_{min}(X)$ consisting
of functionals with compact supports. As a result of this
description, we obtain a characterization of the functionals from
$\mathfrak{R}_{min}(X)_c$ and $\mathfrak{R}_{min}(X)$ (Theorem
\ref{functional}): $\mu\in\mathfrak{R}_{min}(X)_c$ (resp.,
$\mu\in\mathfrak{R}_{min}(X)$) if and only if there exists a
non-empty compact subset $F\subset X$ (resp., $F\subset\beta X$)
such that $\mu(f)=\inf\{f(x):x\in F\}$ (resp., $\mu(f)=\inf\{\beta
f(x):x\in F\}$). A similar characterization holds for the
functionals from $\mathfrak{R}_{max}(X)_c$ and
$\mathfrak{R}_{max}(X)$. Actually, there exists a homeomorphism
$\nu_X\colon \mathfrak{R}_{max}(X)\to\mathfrak{R}_{min}(X)$ such
that
$\nu_X\big(\mathfrak{R}_{max}(X)_c\big)=\mathfrak{R}_{min}(X)_c$.
For any $\mu\in\mathfrak{R}_{max}(X)$ the functional
$\nu_X(\mu)\in\mathfrak{R}_{min}(X)$ is defined by
$\nu_X(\mu)(f)=-\mu(-f)$, $f\in C^*(X)$.

We also establish that for any Tychonoff space $X$ each of the
spaces $\mathfrak{R}_{max}(X)_c$ and $\mathfrak{R}_{min}(X)_c$ is
homeomorphic to the hyperspace $\exp_cX$ of the non-empty compact
subsets of $X$ (see Theorem \ref{exp}) with the Vietoris topology.
Proposition \ref{exp{lsc}} shows that similar results hold for
$\exp_cX$ equipped with the upper or the lower Vietoris topology.
When $X$ is compact, $\mathfrak{R}_{max}(X)_c=\mathfrak{R}_{max}(X)$
and $\mathfrak{R}_{min}(X)_c=\mathfrak{R}_{min}(X)$, so we have a
characterization of the hyperspace $\exp X$ which  was earlier
established by Radul \cite{r0}).

We also prove (see Theorem \ref{retract}) that if $X\subset Y$, then
there exists a continuous compact-valued retraction from $Y$ onto
$X$ iff there exists a normed weakly additive extender $u\colon
C(^*(X)\to C^*(Y)$ with compact supports preserving $\min$ (resp.,
$\max$) and weakly preserving $\max$ (resp., $\min$). Based on
Theorem \ref{retract}, we show (Theorem \ref{ae}) that for any
Tychonoff space $X$ the following conditions are equivalent to $X\in
AE(1)$: (i) For any $C$-embedding of $X$ into a space $Y$ there
exists an extender $u:C^*(X)\rightarrow C^*(Y)$ with compact
supports such that $u$ is normed, weakly additive, preserves $\min$
and weakly preserves $\max$; (ii) For any $C$-embedding of $X$ into
a space $Y$ there exists an extender $u:C^*(X)\rightarrow C^*(Y)$
with compact supports such that $u$ is normed, weakly additive,
preserves $\max$ and weakly preserves $\min$; (iii) For any
$C$-embedding of $X$ into a  space $Y$ there exists a map $\theta
:Y\rightarrow\mathfrak{R}_{min}(X)_c$ such that $\theta(x)=\delta_x$
for all $x$ in $X$; (iv) For any $C$-embedding of $X$ into a  space
$Y$ there exists a map $\theta :Y\rightarrow\mathfrak{R}_{max}(X)_c$
such that $\theta(x)=\delta_x$ for all $x$ in $X$.

In the Section 4 we establish an analogue of Theorem \ref{retract}
concerning upper (resp., lower) semi-continuous compact-valued
retracts. For example, Theorem \ref{usretract} states that the
existence of an upper semi-continuous compact-valued retraction
$r\colon Y\to X$ is equivalent to each of the following conditions:
(i) There exists a normed weakly additive extender $u\colon
C^*(X)\to C^*_{lsc}(Y)$ with compact supports preserving $\min$ and
weakly preserving $\max$; (ii) There exists a normed weakly additive
extender $u\colon C^*(X)\to C^*_{usc}(Y)$ with compact supports
preserving $\max$ and weakly preserving $\min$. Here, $C^*_{lsc}(Y)$
(resp., $C^*_{usc}(Y)$) denotes all bounded lower (resp., upper)
semi-continuous real-valued functions on $Y$. Theorem
\ref{usretract} implies another characterization of $AE(0)$-spaces
in terms of non-linear extenders. In the last section we introduce
the class of {\em Zarichnyi spaces} and raise some questions.

Finally, let us mention that the results in Section 2 are taken from
the first author's MSc thesis \cite{al} which was written under the
supervision of the second author.

%%%%%%%%%%%%%%%%%%%%%%%%%%%%%%%%%%%%%%%%%%%%%%%%%%%
\section{Functionals from $\mathfrak{R}_{max}(X)$ and $\mathfrak{R}_{min}(X)$ and their supports}

Let $\mathfrak{R}_{max}(X)$ (resp., $\mathfrak{R}_{min}(X)$) be the
space of all normed, weakly additive functionals on $C^*(X)$ which
preserve $\max$ and weakly preserve $\min$ (resp., preserve $\min$
and weakly preserve $\max$).

In this section we describe the supports of the functionals from
sets $\mathfrak{R}_{max}(X)$ and $\mathfrak{R}_{min}(X)$. For any
functional $\mu\colon C^*(X)\to\mathbb R$ we define its support
$S(\mu)$ to be the following subset of the \v{C}ech-Stone
compactification $\beta X$ of $X$ (see \cite{v2} for a similar
definition):

\begin{definition}
$S(\mu)$ is the set of all $x\in\beta X$ such that for every
neighborhood $O_x$ of $x$ in $\beta X$ there exist $f,g\in C^*(X)$
with $\beta f|(\beta X\backslash O_x)= \beta g|(\beta X\backslash
O_x)$ and $\mu(f)\neq\mu(g)$.
\end{definition}

Here, $\beta f\colon\beta X\to\mathbb R$ is the \v{C}ech-Stone
extension of $f$ and $\beta f|(\beta X\backslash O_x)$ denotes its
restriction on the set $\beta X\backslash O_x$. Obviously, $S(\mu)$
is a closed subset of $\beta X$ (possibly empty). If
$\varnothing\neq S(\mu)\subset X$, we say that $\mu$ has a compact
support. Identifying $C^*(X)$ with $C(\beta X)$, any functional
$\mu$ on $C^*(X)$ can be considered as a function $\mu\colon C(\beta
X)\to\mathbb R$.

For any $\mu$ let $\mathcal A_\mu$ be the family of all closed
non-empty sets $A\subset\beta X$ such that for any $f,g\in C(\beta
X)$ we have $\mu(f)=\mu(g)$ provided $f|A=g|A$. It is clear that
$\beta X\in\mathcal A_\mu$, so $\mathcal A_\mu\neq\varnothing$.

A functional $\mu$ on $C(\beta X$ is called {\em monotone} if $f\leq
g$ implies $\mu(f)\leq\mu(g)$. Obviously, every functional
preserving $\max$ or $\min$ is monotone.

\begin{lemma}
Suppose that $\mu\colon C(\beta X)\to\mathbb R$ is a normed
functional with $\mu(0_X)=0$. Then $\mathcal A_\mu$ is closed with
respect to finite intersections and $S(\mu)=\bigcap\{A:A\in\mathcal
A_\mu\}$. Moreover, $S(\mu)\in\mathcal A_\mu$ provided $\mu$ is
weakly additive and monotone.
\end{lemma}

\begin{proof}
Suppose $A, B\in\mathcal A_\mu$ with $A\cap B=\varnothing$. There
exists $f\in C(\beta X)$ such that $f(A)=1$ and $f(B)=0$. So,
$f|A=1_{\beta X}|A$ and $f|B=0_{\beta X}|B$. This implies
$\mu(f)=\mu (1_{\beta X})=1$ and $\mu(f)=\mu(0_{\beta X})=0$, a
contradiction. Therefore, $A\cap B\neq\varnothing$ for any two
elements of $\mathcal A_\mu$, and it is easily seen that $A\cap
B\in\mathcal A_\mu$. Then, by induction,
$\bigcap_{i=1}^{i=k}A_i\in\mathcal A_\mu$ if $A_1,..,A_k\in\mathcal
A_\mu$.

For the equality $S(\mu)=\bigcap\{A:A\in\mathcal A_\mu\}$, suppose
$x\not\in S(\mu)$. Then, there exists a neighborhood
$O_x\subset\beta X$ of $x$ such that $\mu(f)=\mu(g)$ for every
$f,g\in C(\beta X)$ with $f|(\beta X\backslash O_x)= g|(\beta
X\backslash O_x)$ (we can assume that $\beta X\backslash
O_x\neq\varnothing$ by choosing a smaller $O_x$). Consequently,
$\beta X\backslash O_x\in\mathcal A_\mu$ and $x\not\in
A_\mu=\bigcap\{A:A\in\mathcal A_\mu\}$. If $x\not\in A_\mu$, there
exists $A\in\mathcal A_\mu$ with $x\not\in A$. Then $O_x=\beta
X\backslash A$ is a neighborhood of $x$ such that $\mu(f)=\mu(g)$
for all $f,g\in C(\beta X)$ with $f|(\beta X\backslash O_x)=
g|(\beta X\backslash O_x)$. Hence, $x\not\in S(\mu)$.

Finally, suppose $\mu$ is weakly additive, and let
$f|S(\mu)=g|S(\mu)$ for some $f,g\in C(\beta X)$. Then, for every
$\epsilon>0$ the set $U_\epsilon=\{x\in\beta
X:|f(x)-g(x)|<\epsilon\}$ is a neighborhood of $S(\mu)$. So, we can
find finitely many $B_1,..,B_j\in\mathcal A_\mu$ such that
$S(\mu)\subset B_0=\bigcap_{i=1}^{i=j}B_i\subset U_\epsilon$. Next,
there exists a function $h\in C(\beta X)$ with $h|B_0=f|B_0$ and
$g(x)-\epsilon\leq h(x)\leq g(x)+\epsilon$ for all $x\in\beta X$.
Indeed, consider the lower semi-continuous convex-valued map
$\Phi\colon\beta X\to\mathbb R$, defined by $\Phi (x)=f(x)$ for
$x\in B_0$ and $\Phi(x)$ to be the interval $[g(x)-\epsilon,
g(x)+\epsilon]$ for $x\not\in B_0$. According to Michael's selection
theorem \cite{m}, $\Phi$ admits a continuous selection $h\in C(\beta
X)$. Since $B_0\in\mathcal A_\mu$, $\mu(f)=\mu(h)$. On the other
hand, the inequalities $g-\epsilon\leq h\leq g+\epsilon$ imply
$\mu(g)-\epsilon\leq\mu(h)\leq\mu(g)+\epsilon$ (recall that $\mu$ is
weakly additive and monotone). Hence, $|\mu(f)-\mu(g)|<\epsilon$ for
every $\epsilon>0$ which yields $\mu(f)=\mu(g)$.
\end{proof}

\begin{cor}
$S(\mu)\in\mathcal A_\mu$ for any normed and weakly additive
monotone functional $\mu$.
\end{cor}

\begin{proof}
This follows from Lemma 2.2 because $1=\mu(0_X+1_X)=\mu(0_X)+1$
implies $\mu(0_X)=0$.
\end{proof}

For any functional $\mu$ on $C(\beta X)$ we denote by $\Lambda_\mu$
the family of all closed subsets $A\subset\beta X$ satisfying the
following condition: if $B\subset\beta X$ is a closed disjoint set
from $A$, then there exists $g\in C(\beta X)$ such that $\mu(g)=0$,
$g(A)\subset (-\infty,0]$ and $g(B)\subset (0,\infty)$. Without loss
of generality, we may assume that $\beta X\in\Lambda_\mu$.

\begin{lemma}
Let $\mu$ be a normed, monotone, weakly additive functional weakly
preserving $\max$ and $\min$. Then,
$\mu(f)=\inf_{A\in\Lambda_\mu}\sup_{x\in A}f(x)$ for any $f\in
C(\beta X)$.
\end{lemma}

\begin{proof}
We follow the proof of Theorem 4.2 from \cite{r1}. Since $\mu$ is
weakly additive, considering the function $f-\mu(f)$ if necessary,
we may assume that $\mu(f)=0$. Then
$A_0=f^{-1}((-\infty,0])\in\Lambda_\mu$ and $\sup_{x\in A_0}f(x)\leq
0$. Suppose there exists $H\in\Lambda_\mu$ such that $\sup_{x\in
H}f(x)<0$, and let $B=f^{-1}([0,\infty))$. According to the
definition of $\Lambda_\mu$, there exists $g\in C(\beta X)$ such
that $\mu(g)=0$, $g(x)\leq 0$ for every $x\in H$ and $g(x)>0$ for
all $x\in B$. Hence, $\min\{0_{\beta X}, f\}(x)<\max\{0_{\beta
X},g\}(x)$ for all $x\in\beta X$. Consequently, $c_{\beta
X}+\min\{0_{\beta X}, f\}\leq\max\{0_{\beta X},g\}$ for some $c>0$.
So, $$\mu(c_{\beta X}+\min\{0_{\beta
X},f\})=c+\min\{0,\mu(f)\}\leq\max\{0,\mu(g)\}.$$ Since
$\min\{0,\mu(f)\}=\max\{0,\mu(g)\}=0$, this a contradiction.
Therefore, $\sup_{x\in H}f(x)\geq 0$ for every $H\in\Lambda_\mu$.
The last inequality together with $\sup_{x\in A_0}f(x)\leq 0$ yields
$\inf_{A\in\Lambda_\mu}\sup_{x\in A}f(x)=0=\mu(f)$.
\end{proof}

If $A\subset\beta X$ is a closed set and $O_A$ its neighborhood in
$\beta X$, let $C(O_A)$ be the set of all functions $f\colon\beta
X\to [-1,0]$ with the following property: there exists an open set
$U\subset\beta X$ such that $A\subset U\subset\overline{U}\subset
O_A$, $f(\overline{U})=-1$ and $f(x)=0$ for all $x\not\in O_A$.

\begin{lemma}
Let $\mu$ be a normed, monotone, weakly additive functional weakly
preserving $\max$ and $\min$. Then $A\in\Lambda_\mu$ if and only if
$\mu(f)\neq 0$ for all $f\in C(O_A)$ and all $O_A$.
\end{lemma}

\begin{proof}
Suppose $A_0\in\Lambda_\mu$ and $f\in C(O_{A_0})$ for some
$O_{A_0}$. Since $\mu(f)=\inf_{A\in\Lambda_\mu}\sup_{x\in A}f(x)$
(see Lemma 2.4) and $f(A_0)=-1$, $\mu(f)\neq 0$.

Now, suppose $A\subset\beta X$ is closed such that $\mu(f)\neq 0$
for all $f\in C(O_A)$ and all $O_A$. To show that $A\in\Lambda_\mu$,
take $B\subset\beta X$ to be a closed set disjoint with $A$. Let
$O_A=X\backslash B$ and $g\in C(O_A)$. Then $g(A)=-1$, $g(B)=0$ and
$\mu(g)\neq 0$. Since $-1_{\beta X}\leq g\leq 0_{\beta X}$, we have
$-1\leq\mu(g)<0$. Hence, $h(x)=-\mu(g)>0$ for all $x\in B$ and
$h(x)=g(x)-\mu(g)=-1-\mu(g)\leq 0$ for all $x\in A$, where
$h=g-\mu(g)$. Therefore, $A\in\Lambda_\mu$.
\end{proof}

\begin{lemma}
Let $\mu$ be a normed, monotone, weakly additive functional weakly
preserving $\max$ and $\min$. Then $s(\mu)=\{x\in\beta
X:\{x\}\in\Lambda_\mu\}$ is a closed subset of $S(\mu)$. Moreover,
$A\cap s(\mu)\neq\varnothing$ for all $A\in\Lambda_\mu$ if $\mu$
preserves $\min$.
\end{lemma}

\begin{proof}
Let us show first that $s(\mu)\subset S(\mu)$. Indeed, otherwise
there exists $x\in s(\mu)\backslash S(\mu)$, and take any $f\in
C(O_x)$, where $O_x=\beta X\backslash S(\mu)$. Since
$\{x\}\in\Lambda_\mu$, by Lemma 2.5, $\mu(f)\neq 0$. On the other
hand, $f\in C(O_x)$ implies $f(x)=0$ for all $x\in S(\mu)$. So,
$f|S(\mu)=0_{\beta X}|S(\mu)$ which yields $\mu(f)=0$ (see Corollary
2.3).

Next, let $x\not\in s(\mu)$. According to Lemma 2.5, there exists a
neighborhood $O_x$ of $x$ and $f\in C(O_x)$ with $\mu(f)=0$. Hence,
$f(\overline{U})=-1$ and $f(\beta X\backslash O_x)=0$ for some open
$U\subset\beta X$ satisfying $x\in U\subset\overline{U}\subset O_x$.
Consequently, $U\cap s(\mu)=\varnothing$ because $f\in C(O_y)$ for
all $y\in U$, where $O_y=O_x$.

Finally, let $\mu$ preserve $\min$, and suppose $A\cap
s(\mu)=\varnothing$ for some $A\in\Lambda_\mu$. Then, for each $x\in
A$ there exist neighborhoods $O_x$ and $U_x$ of $x$ and $f\in
C(O_x)$ such that  $x\in U_x\subset\overline{U_x}\subset O_x$,
$f_x\in C(O_x)$, $f_x(\overline{U_x})=-1$, $f_x(\beta X\backslash
O_x)=0$ and $\mu(f_x)=0$. Take finitely many points $x_1,..,x_k\in
A$ with $A\subset U=\bigcup_{i=1}^{i=k}U_{x_i}$. Let
$f=\min\{f_{x_i}:i\leq k\}$ and $O_A=\bigcup_{i=1}^{i=k}O_{x_i}$.
Then $\mu(f)=\min\{\mu(f_{x_i}):i\leq k\}=0$. On the other hand, we
have $A\subset U\subset\overline{U}\subset O_A$,
$f(\overline{U})=-1$ and $f(x)=0$ for all $x\not\in O_A$. So, $f\in
C(O_A)$ which, according to Lemma 2.5, implies $\mu(f)\neq 0$. This
contradiction completes the proof.
\end{proof}

\begin{cor}
Let $\mu$ be a normed, weakly additive functional weakly preserving
$\max$ and preserving $\min$. Then $s(\mu)=S(\mu)$ and
$\mu(f)=\inf\{f(x):x\in S(\mu)\}$ for all $f\in C(\beta X)$.
\end{cor}

\begin{proof}
We show first that $\inf_{A\in\Lambda_\mu}\sup_{x\in
A}f(x)=\inf\{f(x):x\in s(\mu)\}$ for any $f\in C(\beta X)$. Indeed,
since every $x\in s(\mu)$ belongs to $\Lambda_\mu$, we have
$\inf_{A\in\Lambda_\mu}\sup_{x\in A}f(x)\leq\inf\{f(x):x\in
s(\mu)\}$. The reverse inequality follows from the fact that every
$A\in\Lambda_\mu$ intersect $s(\mu)$ (Lemma 2.6). Hence, by Lemma
2.4, $\mu(f)=\inf_{A\in\Lambda_\mu}\sup_{x\in A}f(x)=\inf\{f(x):x\in
s(\mu)\}$.

To complete the proof, it suffices to show that $s(\mu)=S(\mu)$.
Suppose $f|s(\mu)=g|s(\mu)$ for some $f,g\in C(\beta X)$. Then,
$\inf\{f(x):x\in s(\mu)\}=\inf\{g(x):x\in s(\mu)\}$, so
$\mu(f)=\mu(g)$. This means that $s(\mu)\in\mathcal A_\mu$. But,
$S(\mu)$ is the smallest element of $\mathcal A_\mu$ (Lemma 2.2).
Therefore, $s(\mu)=S(\mu)$
\end{proof}

Concerning the functionals $\mu\in \mathfrak{R}_{max}(X)$, their
supports have the following property.

\begin{proposition}
Let $\mu$ be a normed, weakly additive functional weakly preserving
$\min$ and preserving $\max$. Then $\mu(f)=\sup\{f(x):x\in S(\mu)\}$
for all $f\in C(\beta X)$.
\end{proposition}

\begin{proof}
The proof follows from the fact that the map
$\nu_X\colon\mathfrak{R}_{max}(X)\to\mathfrak{R}_{min}(X)$,
$\nu_X(\mu)(f)=-\mu(-f)$, is a homeomorphism. Indeed, if
$\mu\in\mathfrak{R}_{max}(X)$, then the functional
$\nu=\nu_X(\mu)\in\mathfrak{R}_{min}(X)$ and, according to Corollary
2.7, $\nu(f)=\inf\{f(x):x\in S(\nu)\}$ for any $f\in C(\beta X)$.
Consequently, $\mu(f)=\sup\{f(x):x\in S(\nu)\}$. The last equality
implies that $S(\nu)$ is the support of $\mu$, which completes the
proof.
\end{proof}

We complete this section with the following characterization of the
functionals from $\mathfrak{R}_{min}(X)\cup\mathfrak{R}_{max}(X)$.

\begin{theorem}\label{functional}
Let $X$ be a Tychonoff space and $\mu$ a functional on $C^*(X)$.
Then we have:
\begin{itemize}
\item[(i)] $\mu\in\mathfrak{R}_{min}(X)_c$ $($resp., $\mu\in\mathfrak{R}_{min}(X))$  if
and only if there exists a non-empty compact set $F\subset X$
$($resp., $F\subset\beta X)$ such that $F=S(\mu)$ and
$\mu(f)=\inf\{f(x):x\in F\}$ for all $f\in C(\beta X)$;
\item[(ii)] $\mu\in\mathfrak{R}_{max}(X)_c$ $($resp., $\mu\in\mathfrak{R}_{max}(X))$ if
and only if there exists a non-empty compact set $F\subset X$
$($resp., $F\subset\beta X)$ such that $F=S(\mu)$ and
$\mu(f)=\sup\{f(x):x\in F\}$ for all $f\in C(\beta X)$.
\end{itemize}
\end{theorem}

\begin{proof}
We are going to proof the first item only, the proof of the second
one is similar. If $\mu\in\mathfrak{R}_{min}(X)_c$ $($resp.,
$\mu\in\mathfrak{R}_{min}(X))$, then $F=S(\mu)$ is a non-empty
compact subset of $X$ (resp., $\beta X$) and, by Corollary 2.7,
$\mu(f)=\inf\{f(x):x\in F\}$, $f\in C(\beta X)$. Suppose there
exists a compact $F\subset X$ (resp., $F\subset\beta X$) with
$\mu(f)=\inf\{f(x):x\in F\}$ for all $f\in C(\beta X)$. It is easily
seen that $\mu\in\mathfrak{R}_{min}(X)$ and Corollary 2.7 implies
$F=S(\mu)$. Moreover, $\mu\in\mathfrak{R}_{min}(X)_c$ provided
$F\subset X$.
\end{proof}

%%%%%%%%%%%%%%%%%%%%%%%%%%%%%%%%%%%%%%%%%%%%%%%%%%%%%%%%%%%%%%%%%%%%%%%%%%%%%%%%%%%%%%%%%%%%%%%%%%%%%%%%
\section{Set-valued continuous retractions and $AE(1)$-spaces}

Below, by $\exp\beta X$ we denote all closed non-empty subsets of
$\beta X$ with the Vietoris topology, and by $\exp_cX$ the subspace
of $\exp\beta X$ consisting of all compact subsets of $X$.

Recall that a set-valued map $\varphi\colon X\to Y$ between two
spaces is called lower (resp., upper) semi-continuous if the set
$\{x\in X:r(x)\cap U\neq\varnothing\}$ (resp., $\{x\in X:r(x)\subset
U\}$) is open in $X$ for every open $U\subset Y$. When $\varphi$ is
both lower and upper semi-continuous, then it is called continuous.
We also say that $\varphi$ is compact-valued if $\varphi(x)$ is a
non-empty compactum for each $x\in X$.

\begin{theorem}\label{exp}
Let $X$ be a Tychonoff space. Then
\begin{itemize}
\item[(i)] Each of the spaces $\mathfrak{R}_{min}(X)$ and $\mathfrak{R}_{max}(X)$
is homeomorphic to $\exp\beta X$;
\item[(ii)] Each of the spaces $\mathfrak{R}_{min}(X)_c$ and
$\mathfrak{R}_{max}(X)_c$ is homeomorphic to $\exp_cX$.
\end{itemize}
\end{theorem}

\begin{proof}
We are going to prove only item (i), the proof of (ii) is similar.
First, observe that $\mathfrak{R}_{min}(X)$ is a compact subspace of
the product $\mathbb R^{C^*(X)}$. Indeed, Theorem 2.9 implies that
$\mathfrak{R}_{min}(X)$ is a subset of the compact product
$K=\prod\{[a_f,b_f]:f\in C^*(X)\}$, where $a_f=\inf\{f(x):x\in X\}$
and $b_f=\sup\{f(x):x\in X\}$). Moreover, if $\{\mu_\alpha\}$ is a
net in $\mathfrak{R}_{min}(X)$ converging to some $\mu\in K$, then
$\{\mu_\alpha(f)\}$ converges to $\mu(f)$ for all $f\in C^*(X)$.
This yields that $\mu\in\mathfrak{R}_{min}(X)$.

Consider the set-valued map $\Phi\colon\mathfrak{R}_{min}(X)\to\beta
X$, $\Phi(\mu)=S(\mu)$. Obviously, $\Phi(\delta_x)=\{x\}$ for all
$x\in\beta X$. Next, we are going to show that the map $\Phi$ is
lower semi-continuous. Suppose $S(\mu_0)\cap U\neq\varnothing$ for
some $\mu_0\in\mathfrak{R}_{min}(X)$ and open $U\subset\beta X$.
Take $x_0\in S(\mu_0)\cap U$ and a function $g\in C(\beta X)$ with
$g(x_0)=-1$ and $g(\beta X\backslash U)=1$. Then, by Corollary 2.7,
$\mu_0(g)=\inf\{g(x):x\in S(\mu_0)\}\leq -1$. Hence, the set
$V=\{\mu\in\mathfrak{R}_{min}(X):\mu(g)<0\}$ is a neighborhood of
$\mu_0$ in $\mathfrak{R}_{min}(X)$. For every $\mu\in V$ we have
$S(\mu)\cap U\neq\varnothing$ (otherwise $S(\mu)\subset \beta
X\backslash U$ and $\mu(g)=\inf\{g(x):x\in S(\mu)\}=1$, a
contradiction). Therefore, $\Phi$ is lower semi-continuous.

Assume now that $\mu_0\in\mathfrak{R}_{min}(X)$ and $S(\mu_0)\subset
U$ with $U\subset\beta X$ being open. Choose $h\in C(\beta X)$ such
that $h(S(\mu_0))=0$ and $h(\beta X\backslash U)=-1$. Then
$W=\{\mu\in\mathfrak{R}_{min}(X):\mu(h)>-1/2\}$ is a neighborhood of
$\mu_0$ and $S(\mu)\subset U$ for all $\mu\in W$. So, $\Phi$ is
upper semi-continuous.

Since $\Phi$ is both lower semi-continuous and upper
semi-continuous, it is continuous considered as a single-valued map
from $\mathfrak{R}_{min}(X)$ into $\exp\beta X$. $\Phi$ is also
one-to-one. Indeed, if $\Phi(\mu_1)=\Phi(\mu_2)$, then
$\mu_1(f)=\mu_2(f)$ (see Corollary 2.7) for every $f\in C(\beta X)$.
So, $\mu_1=\mu_2$.

Finally, let us show that $\Phi$ is surjective. For every
$F\in\exp\beta X$ we define the functional $\mu_F\colon
C^*(X)\to\mathbb R$, $\mu_F(f)=\inf\{\beta f(x):x\in F\}$. It is
easily seen that $\mu_F\in\mathfrak{R}_{min}(X)$. It suffices to
prove that $S(\mu_F)=F$. If there exists $a\in S(\mu_F)\backslash
F$, we take $g\in C(\beta X)$ such that $g(a)=0$ and $g(F)=1$. The
last equality implies $\mu_F(g)=1$. On the other hand, by Corollary
2.7, $\mu_F(g)\leq 0$. Similarly, we can obtain a contradiction if
$F\backslash S(\mu_F)\neq\varnothing$. Hence, $\Phi$ is a
homeomorphism between $\mathfrak{R}_{min}(X)$ and $\exp\beta X$.
Since $S(\mu)\subset X$ for all $\mu\in\mathfrak{R}_{min}(X)_c$, it
also follows that $\Phi$ is a homeomorphism from
$\mathfrak{R}_{min}(X)_c$ onto $\exp_cX$.
\end{proof}

We denote by $\mathfrak{R}_{min}^{lsc}(X)_c$ (resp.,
$\mathfrak{R}_{max}^{lsc}(X)_c$) the set $\mathfrak{R}_{min}(X)_c$
(resp., $\mathfrak{R}_{max}(X)_c$) with the topology generated by
the family $\{\mu: \mu(f_i)>a_i, i=1,..,k\}$, where $f_i\in C^*(X)$
and $a_i\in\mathbb R$. Similarly, $\mathfrak{R}_{min}^{usc}(X)_c$
(resp., $\mathfrak{R}_{max}^{usc}(X)_c$) is the set
$\mathfrak{R}_{min}(X)_c$ (resp., $\mathfrak{R}_{max}(X)_c$) with
the topology generated by the family $\{\mu: \mu(f_i)<a_i,
i=1,..,k\}$. Moreover, $\exp_c^+X$ and $\exp_c^{-}X$ denote the set
$\exp_cX$ with the upper (resp., lower) Vietoris topology. Recall
that the upper (resp., lower) Vietoris topology on $\exp_cX$ is the
topology generated by the families $\{F\in\exp_cX:F\subset U\}$
(resp., $\{F\in\exp_cX: F\cap U\neq\varnothing\}$), where $U\subset
X$ is open.

Following the proof of Theorem \ref{exp}, one can establish the
following proposition.

\begin{proposition}\label{exp{lsc}}
Let $X$ be a Tychonoff space. Then
\begin{itemize}
\item[(i)] Each of the spaces $\mathfrak{R}_{min}^{lsc}(X)_c$ and $\mathfrak{R}_{max}^{usc}(X)_c$
is homeomorphic to $\exp_c^+ X$;
\item[(ii)] Each of the spaces $\mathfrak{R}_{min}^{usc}(X)_c$ and
$\mathfrak{R}_{max}^{lsc}(X)_c$ is homeomorphic to $\exp_c^{-}X$.
\end{itemize}
\end{proposition}

Next results provides a connection between continuous set-valued
retractions and extenders (recall that a continuous set-valued map
means a set-valued map which is both lower and upper
semi-continuous).

\begin{theorem}\label{retract}
Let $X$ be a subspace of $Y$. Then the following conditions are
equivalent:
\begin{itemize}
\item[(i)] There exists a continuous compact-valued map $r\colon Y\to\beta X$ with
$r(x)=\{x\}$ for all $x\in X$;
\item[(ii)] There exists an extender
$u:C^*(X)\rightarrow C^*(Y)$ which is normed, weakly additive,
preserves $\min$ and weakly preserves $\max$;
\item[(iii)] There exists an extender
$u:C^*(X)\rightarrow C^*(Y)$ which is normed, weakly additive,
preserves $\max$ and weakly preserves $\min$.
\end{itemize}
Moreover, there exists a continuous compact-valued retraction
$r\colon Y\to X$ iff the extenders from $($ii$)$ and $($iii$)$ have
compact supports.
\end{theorem}

\begin{proof}
Suppose $r\colon Y\to\beta X$ is a continuous compact-valued map
with $r(x)=\{x\}$ for all $x\in X$. Then, for every $f\in C^*(X)$
the equality $u(f)(y)=\inf\{\beta f(x):x\in r(y)\}$, $y\in Y$,
defines a function $u(f)\colon Y\to\mathbb R$. Since $r$ is both
lower and upper semi-continuous, each $u(f)$, $f\in C^*(X)$, is
continuous. Moreover, $u(f)(x)=f(x)$ provided $x\in X$. So, $u$ is
an extender, and one can check that it is normed, weakly additive,
preserves $\min$ and weakly preserves $\max$. Hence, $(i)$ implies
$(ii)$. The implication $(i)\Rightarrow (iii)$ is similar, we define
the extender $u$ by $u(f)(y)=\sup\{\beta f(x):x\in r(y)\}$, where
$f\in C^*(X)$ and  $y\in Y$.

It is easily seen that if $r\colon Y\to X$ is a continuous
compact-valued retraction, then $u(f)(y)=\inf\{f(x):x\in r(y)\}$
(resp., $u(f)(y)=\sup\{f(x):x\in r(y)\}$) defines a normed and
weakly additive extender $u\colon C^*(X)\to C^*(Y)$ with compact
supports such that $u$ preserves $\min$ (resp., $\max$) and weakly
preserves $\max$ (resp., $\min$).

To prove the implication $(ii)\Rightarrow (i)$ (resp.,
$(iii)\Rightarrow (i)$), let $\theta\colon
Y\to\mathfrak{R}_{min}(X)$ (resp., $\theta\colon
Y\to\mathfrak{R}_{max}(X)$) be the map $\theta(y)=\mu_y$. Here
$\mu_y\colon C^*(X)\to\mathbb R$ are the functionals generated by
the extender $u$, i.e., $\mu_y(f)=u(f)(y)$ for all  $f\in C^*(X)$
and $y\in Y$. It follows from the last equality that $\theta$ is
continuous. Moreover, the compact-valued map assigning to each
$\mu_y$ its support $S(\mu_y)$ is lower and upper semi-continuous
(see the proof of Theorem \ref{exp}). So, $r(y)=S(\mu_y)$ defines a
continuous compact-valued map from $Y$ into $\beta X$. Since
$\mu_x=\delta_x$ for any $x\in X$, $r(x)$ is the point-set $x$.

If the extender $u$ from items $(ii)$ and $(iii)$ has compact
supports, then $S(\mu_y)\subset X$ for all $y\in Y$. Hence, in this
case $r$ is a continuous compact-valued retraction from $Y$ onto
$X$.
\end{proof}

We are now in a position to prove the characterization of
$AE(1)$-spaces mentioned in the introduction. We recall the
definition of absolute extensors for $n$-dimensional spaces (br.,
$AE(n)$) in the class of Tychonoff spaces (see \cite{ch1}): $X\in
AE(n)$ if any map $g\colon Z_0\to X$, where $Z_0$ is a subset of a
space $Z$ with $\dim Z\leq n$ and $C(g)(C(X))\subset C(Z)|Z_0$, can
be extended to a map $\overline{g}\colon Z\to X$. Here,
$C(g)(C(X))\subset C(Z)|Z_0$ means that for every function $h\in
C(X)$ the composition $h\circ g$ is extendable over $Z$. In
particular, this is true if $Z$ is norma and $Z_0\subset Z$ closed.

\begin{theorem}\label{ae}
For any space $X$ the following conditions are equivalent:
\begin{itemize}
\item[(i)] $X\in AE(1)$;
\item[(ii)] For any $C$-embedding of $X$ into a space $Y$ there exists an extender
$u:C^*(X)\rightarrow C^*(Y)$ with compact supports such that $u$ is
normed, weakly additive, preserves $\min$ and weakly preserves
$\max$;
\item[(iii)] For any $C$-embedding of $X$ into a space $Y$ there
exists an extender $u:C^*(X)\rightarrow C^*(Y)$ with compact
supports such that $u$ is normed, weakly additive, preserves $\max$
and weakly preserves $\min$;
\item[(iv)] For any $C$-embedding of $X$ into a  space $Y$ there exists a continuous map $\theta
:Y\rightarrow\mathfrak{R}_{min}(X)_c$ such that $\theta(x)=\delta_x$
for all $x$ in $X$;
\item[(v)] For any $C$-embedding of $X$ into a  space $Y$ there exists a continuous map $\theta
:Y\rightarrow\mathfrak{R}_{max}(X)_c$ such that $\theta(x)=\delta_x$
for all $x$ in $X$.
\end{itemize}
\end{theorem}

\begin{proof}
Observe that $(ii)\Leftrightarrow (iii)$ and $(iv)\Leftrightarrow
(v)$. The first equivalence follows from the fact that an operator
$u:C^*(X)\rightarrow C^*(Y)$ is a normed, weakly additive extender
which preserves $\max$ and weakly preserves $\min$ iff the operator
$v:C^*(X)\rightarrow C^*(Y)$, $v(f)=-u(-f)$, is a normed, weakly
additive extender which preserves $\min$ and weakly preserves
$\max$. Concerning the second equivalence, observe that a map
$\theta :Y\rightarrow\mathfrak{R}_{max}(X)_c$ is continuous with
$\theta(x)=\delta_x$ for all $x$ in $X$ if and only if the map
$\theta^{'}:Y\rightarrow\mathfrak{R}_{min}(X)_c$,
$\theta^{'}(y)=\nu_X(\theta(y))$, is continuous and
$\theta^{'}(x)=\delta_x$ for all $x$ in $X$. Here,
$\nu_X\colon\mathfrak{R}_{max}(X)_c\to\mathfrak{R}_{min}(X)_c$ is
the homeomorphism considered above.

So, it suffices to prove the implications $(i)\Rightarrow
(ii)\Rightarrow (iv)\Rightarrow (i)$. Suppose $X\in AE(1)$ and $X$
is $C$-embedded in a space $Y$. Considering $Y$ as a $C$-embedded
subset of the product $\mathbb R^\tau$ for some cardinal $\tau$, we
may assume that $Y=\mathbb R^\tau$. Following the proof of
implication $(i)\Rightarrow (ii)$ of Theorem 3.9 from \cite{cv}, we
embed $\mathbb R^\tau$ as a dense subset of $\mathbb I^\tau$
($\mathbb I=[0,1]$) and let $g\colon T\to\mathbb I^\tau$ be an open
monotone surjection with $T$ being an $AE(0)$-compactum of dimension
one (such $T$ exists by \cite[Theorem 9]{dr1}). Since $g$ is open,
$K=g^{-1}(\mathbb R^\tau)$ is dense in $T$. Hence, by
\cite[Corollary 7]{ch2}, $\dim K=\dim T=1$. Let $K_0=g^{-1}(X)$ and
$g_0=g|Z_0$. Because $X$ is $C$-embedded in $\mathbb R^\tau$, it is
easily seen that $C(g_0)(C(X))\subset C(K)|K_0$. Therefore, the map
$g_0$ can be extended to a map $h\colon K\to X$ (recall that $X\in
AE(1)$). Then the compact-valued map $r\colon\mathbb R^\tau\to X$,
defined by $r(y)=h(g^{-1}(y))$, is both lower semi-continuous
(because $g$ is open) and upper semi-continuous (because $g$ is
closed). Moreover, since each fiber $g^{-1}(y)$, $y\in\mathbb
R^\tau$, is a continuum, so are the values of $r$. Finally, observe
that  $r(x)=\{x\}$ for all $x\in X$. Then, according to Theorem
\ref{retract}, there exists a normed weakly additive extender
$u\colon C^*(X)\to C^*(\mathbb R^\tau)$ with compact supports such
that $u$ preserves $\min$ and weakly preserves $\max$. This
completes the proof of the implication $(i)\Rightarrow (ii)$.

The implication $(ii)\Rightarrow (iv)$ follows from the proof of
Theorem \ref{retract}. Indeed, the map $\theta
:Y\rightarrow\mathfrak{R}_{min}(X)_c$, $\theta(y)=\mu_y$, is the
required one.

To prove the last implication $(iv)\Rightarrow (i)$, consider $X$ as
a $C$-embedded subset of some $\mathbb R^\tau$, and let
$\theta\colon\mathbb R^\tau\to\mathfrak{R}_{min}(X)_c$ be a
continuous map with $\theta(x)=\delta_x$, $x\in X$. It was
established in the proof of Theorem \ref{retract} that the equality
$r(y)=S(\theta(y))$, $y\in\mathbb R^\tau$,  defines a continuous
compact-valued retraction from $\mathbb R^\tau$ onto $X$. For every
$y\in\mathbb R^\tau$, let $F(y)$ be the closure in $\mathbb R^\tau$
of the convex hull $conv(r(y))$. Since $r(y)$ is compact, $F(y)$ is
a convex compact subset of $\mathbb R^\tau$. Finally, let
$\varphi(y)=r(F(y))$, $y\in\mathbb R^\tau$. It is easily seen that
$\varphi\colon\mathbb R^\tau\to X$ is upper semi-continuous. Since
$r$ is continuous and compact-valued, each $\varphi (y)$ is a
continuum. Hence, $\varphi$ is an upper semi-continuous
continuum-valued retraction from $\mathbb R^\tau$ onto $X$.
Therefore, by \cite[Theorem 3.9(ii)]{cv}, $X\in AE(1)$.
\end{proof}

\begin{cor}
A space $X$ is an $AE(1)$ if and only if for every $C$-embedding of
$X$ into a space $Y$ there exists an extender $u:C(X)\rightarrow
C(Y)$ with compact supports such that $u$ is normed, weakly
additive, preserves $\min$ and weakly preserves $\max$ $($resp.,
preserves $\max$ and weakly preserves $\min$$)$.
\end{cor}

\begin{proof}
Suppose $X\in AE(1)$. As in the proof of Theorem \ref{ae} we can
assume that $Y$ is a subset of $\mathbb R^\tau$ for some $\tau$ and
there exists a  continuous compact-valued retraction $r\colon\mathbb
R^\tau\to X$. Then $u(f)(y)=\inf\{f(x):x\in r(y)\}$ (resp.,
$u(f)(y)=\sup\{f(x):x\in r(y)\}$) defines the required extender
$u\colon C(X)\to C(\mathbb R^\tau)$. The other direction follows
directly from Theorem \ref{ae} because for any monotone normed and
weakly additive extender $u\colon C(X)\to C(\mathbb R^\tau)$ we have
$u(C^*(X))\subset C^*(\mathbb R^\tau)$.
\end{proof}

%%%%%%%%%%%%%%%%%%%%%%%%%%%%%%%%%%%%%%%%%%%%%%%%%%%%%%%%%%%%%%%%%%%%%%
\section{Upper and lower semi-continuous retractions}

In this section we describe a connection between upper (resp.,
lower) semi-continuous retractions and functional extenders. Recall
that a function $f\colon X\to\mathbb R$ is called lower (resp.,
upper) semi-continuous if $f^{-1}(a,\infty)$ (resp.,
$f^{-1}(-\infty,a)$) is open in $X$ for every $a\in\mathbb R$. For
any space $X$ we denote by $C^*_{lsc}(X)$ (resp., $C^*_{usc}(X)$)
the set of all bounded lower (resp., upper) semi-continuous
functions on $X$.

The following theorem is an analogue of Theorem \ref{retract}.

\begin{theorem}\label{usretract}
Let $X$ be a subspace of $Y$. Then the following conditions are
equivalent:
\begin{itemize}
\item[(i)] There exists an upper semi-continuous compact-valued map\\ $r\colon Y\to\beta X$ with
$r(x)=\{x\}$ for all $x\in X$;
\item[(ii)] There exists an extender
$u:C^*(X)\rightarrow C^*_{lsc}(Y)$ which is normed, weakly additive,
preserves $\min$ and weakly preserves $\max$;
\item[(iii)] There exists an extender
$u:C^*(X)\rightarrow C^*_{usc}(Y)$ which is normed, weakly additive,
preserves $\max$ and weakly preserves $\min$.
\end{itemize}
Moreover, the extenders from $($ii$)$ and $($iii$)$ have compact
supports iff $r(y)\subset X$ for all $y\in Y$.
\end{theorem}

\begin{proof}
Let $r\colon Y\to\beta X$ be a compact-valued upper semi-continuous
map with $r(x)=\{x\}$ for all $x\in X$. Then, for every $f\in
C^*(X)$, the equality $u(f)(y)=\inf\{\beta f(x):x\in r(y)\}$, $y\in
Y$, defines a bounded function $u(f)\colon Y\to\mathbb R$.
Obviously, $u(f)(x)=f(x)$ provided $x\in X$. So, $u$ is an extender,
and it is easily seen that $u$ is normed, weakly additive, preserves
$\min$ and weakly preserves $\max$. Using that $r$ is upper
semi-continuous, one can show that $u(f)\in C^*_{lsc}(Y)$. Indeed,
suppose $u(f)(y_0)>a$ for some $y_0\in Y$ and $a\in\mathbb R$. Then
$r(y_0)\subset (a,\infty)$ and there exists a neighborhood
$O(y_0)\subset Y$ of $y_0$ such that $r(y)\subset (a,\infty)$ for
all $y\in O(y_0)$. Hence, $u(f)(y)>a$, $y\in O(y_0)$.  Therefore,
$(i)$ implies $(ii)$. The implication $(i)\Rightarrow (iii)$ is
similar, we define the extender $u$ by $u(f)(y)=\sup\{\beta
f(x):x\in r(y)\}$, where  $f\in C^*(X)$ and  $y\in Y$. In this case
$u(f)\in C^*_{usc}(Y)$.

Suppose $r\colon Y\to X$ is an upper semi-continuous compact-valued
retraction.  Then each of the extenders $u$ defined above has
compact supports. Indeed, by Theorem 2.9, the support of any
functional $\mu_y$, $\mu_y(f)=u(f)(y)$, is the set $r(y)$.

The implication $(ii)\Rightarrow (i)$ follows from the observation
that if $u:C^*(X)\rightarrow C^*_{lsc}(Y)$ is normed, weakly
additive extender  which  preserves $\min$  and weakly preserves
$\max$, then the functionals $\mu_y$, $y\in Y$, belong to
$\mathfrak{R}_{min}(X)$. So, by Theorem 2.9, $S(\mu_y)$ is a
non-empty compact subset of $\beta X$. Moreover, $S(\mu_x)=\{x\}$
for all $x\in X$ because $u$ is an extender. Hence, the set-valued
map $r\colon Y\to\beta X$, $r(y)=S(\mu_y)$, is  compact-valued with
$r(x)=\{x\}$ for $x\in X$. Let us show that $r$ is upper
semi-continuous. Suppose $r(y_0)\subset U$ for some $y_0\in Y$ and
open $U\subset\beta X$. Take a function $f\in C(\beta X)$ with
$f(x)=1$ for all $x\in r(y_0)$ and $f(\beta X\backslash U)=0$. Since
$u(f)$ is lower semi-continuous, $y_0$ has a neighborhood $V(y_0)$
such that $u(f)(y)>1/2$ for all $y\in V(y_0)$. This implies
$r(y)\subset U$, $y\in V(y_0)$. Indeed, otherwise Corollary 2.7
would yield $\mu_y(f)=u(f)(y)=0$ for some $y\in V(y_0)$. Obviously,
$r(y)\subset X$ when $u$ has compact supports. Similar arguments
provide the proof of $(iii)\Rightarrow (i)$.
\end{proof}

We can establish now a characterization of $AE(0)$-spaces in terms
of normed weakly additive extenders with compact supports preserving
$\min$ (resp., $\max$) and weakly preserving $\max$ (resp., $\min$).

\begin{cor}\label{ae(0)}
The following conditions are equivalent for any space $X$:
\begin{itemize}
\item[(i)] $X\in AE(0)$;
\item[(ii)] For every $C$-embedding of $X$ in a space $Y$ there exists a normed
weakly additive extender $u:C^*(X)\rightarrow C^*_{lsc}(Y)$ with
compact supports which preserves $\min$ and weakly preserves $\max$;
\item[(iii)] For every $C$-embedding of $X$ in a space $Y$ there exists a normed
weakly additive extender $u:C^*(X)\rightarrow C^*_{usc}(Y)$ with
compact supports which preserves $\max$ and weakly preserves $\min$.
\end{itemize}
\end{cor}

\begin{proof}
The proof follows from Theorem 4.1 and the following
characterization of $AE(0)$-spaces \cite{v3}: $X\in AE(0)$ if and
only if for any $C$-embedding of $X$ in a space $Y$ there exists a
compact-valued upper semi-continuous retraction $r\colon Y\to X$.
\end{proof}

Next theorem shows that conditions (ii) and (iii) from Theorem
\ref{usretract} can be weakened.

\begin{theorem}\label{m}
Let $X$ be a subspace of $Y$. Then the following conditions are
equivalent:
\begin{itemize}
\item[(i)] There exists an upper semi-continuous compact-valued map\\ $r\colon Y\to\beta X$ with
$r(x)=\{x\}$ for all $x\in X$;
\item[(ii)] There exists an extender
$u:C^*(X)\rightarrow C^*_{lsc}(Y)$ preserving $\min$ with
$u(1_X)=1_Y$;
\item[(iii)] There exists an extender
$u:C^*(X)\rightarrow C^*_{usc}(Y)$ preserving $\max$ with
$u(1_X)=1_Y$.
\end{itemize}
\end{theorem}

\begin{proof}
According to Theorem 4.1, it suffices to show the implications
$(ii)\Rightarrow (i)$ and $(iii)\Rightarrow (i)$.  We are going to
prove first $(iii)\Rightarrow (i)$. Suppose $u:C^*(X)\rightarrow
C^*_{usc}(Y)$ is an extender preserving $\max$ and $u(1_X)=1_Y$. For
every open set $U\subset X$ let
$$\mathrm{e}(U)=\bigcup\{u(h)^{-1}((-\infty,1)):h\in C_U\},$$
where $C_U$ is the set of all $h\in C^*(X)$ such that $h(X)\subset
(-\infty,1]$ and $X\backslash U\subset h^{-1}(1)$. Since $u$
preserves $\max$, it is monotone. Hence, $u(h)\leq u(1_X)=1_Y$ for
all $h\in C_U$. Because each $u(h)$ is upper semi-continuous,
$u(h)^{-1}((-\infty,1))$ is open in $Y$, so is the set
$\mathrm{e}(U)$. Using that $u$ is an extender, one can show that
$\mathrm{e}(U)\cap X=U$. Moreover, if $U\subset V$, then $C_U\subset
C_V$ and we have $\mathrm{e}(U)\subset\mathrm{e}(V)$.

We claim that $\mathrm{e}(U\cap V)=\mathrm{e}(U)\cap\mathrm{e}(V)$
for any two open sets $U,V\subset X$. Indeed, the inclusion
$\mathrm{e}(U\cap V)\subset\mathrm{e}(U)\cap\mathrm{e}(V)$ follows
from monotonicity of the operator $\mathrm{e}$. To prove the other
inclusion, let $y\in \mathrm{e}(U)\cap\mathrm{e}(V)$. Then there
exist $h_U\in C_U$ and $h_V\in C_V$ with $u(h_U)(y)<1$ and
$u(h_V)(y)<1$. Obviously, $h_U^{-1}((-\infty,1))\cap
h_V^{-1}((-\infty,1))=\varnothing$ implies $\max\{h_U,h_V\}=1_X$.
So, $u(\max\{h_U,h_V\})(y)=\max\{u(h_U)(y),u(h_V)(y)\}=1$, which is
a contradiction. Therefore, $h_U^{-1}((-\infty,1))\cap
h_V^{-1}((-\infty,1))\neq\varnothing$. On the other hand,
$g=\max\{h_U,h_V\}$ belongs to $C_{U\cap V}$ and $y\in
g^{-1}((-\infty,1))$. Thus, $y\in\mathrm{e}(U\cap V)$.

Now, we define the set-valued map $r\colon Y\to\beta X$ by
$$r(y)=\bigcap\{\overline{U}^{\beta
X}:y\in\mathrm{e}(U)\}\hbox{~}\mbox{if}\hbox{~}
y\in\bigcup\{\mathrm{e}(U):U\in\mathcal T_X\}$$ and $$r(y)=\beta
X\hbox{~}\mbox{if}\hbox{~}
y\not\in\bigcup\{\mathrm{e}(U):U\in\mathcal T_X\}.$$ Using that
$\bigcap_{i=1}^{i=k}\mathrm{e}(U_i)=\mathrm{e}(\bigcap_{i=1}^{i=k}U_i)$
for any finitely many open sets $U_i\subset X$, one can show that
$r$ is an upper semi-continuous map with non-empty values. Since
$\mathrm{e}(U)\cap X=U$, $U\in\mathcal T_X$, we have $r(x)=\{x\}$
for all $x\in X$.

The proof of $(ii)\Rightarrow (i)$ is similar. The only difference
is the definition of the operator $\mathrm{e}$. Now we define
$$\mathrm{e}(U)=\bigcup\{u(h)^{-1}((1,\infty)):h\in C_U\},$$ where
$C_U$ is the set of all $h\in C^*(X)$ such that $h(X)\subset
[1,\infty)$ and $X\backslash U\subset h^{-1}(1)$.
\end{proof}

Next corollary follows from Theorem 4.3 and Dranishnikov's
characterization \cite{dr} of compact $AE(0)$-spaces as upper
semi-continuous retracts of Tychonoff cubes.

\begin{cor}\label{cptae(0)}
The following conditions are equivalent for a compact space $X$:
\begin{itemize}
\item[(i)] $X\in AE(0)$;
\item[(ii)] For every embedding of $X$ in a space $Y$ there exists an
extender $u:C^*(X)\rightarrow C^*_{lsc}(Y)$ which preserves $\min$
and $u(1_X)=1_Y$;
\item[(iii)] For every embedding of $X$ in a space $Y$ there exists an
extender $u:C^*(X)\rightarrow C^*_{usc}(Y)$  which preserves $\max$
and $u(1_X)=1_Y$.
\end{itemize}
\end{cor}

Observe that conditions $(ii)$ and $(iii)$ are equivalent. Indeed,
if $u:C^*(X)\rightarrow C^*_{lsc}(Y)$ is an extender preserving
$\min$ and $u(1_X)=1_Y$, then the formula $v(h)=-u(-h)$ defines an
extender $v:C^*(X)\rightarrow C^*_{usc}(Y)$ which preserves $\max$
and $v(1_X)=1_Y$. Similarly, condition $(iii)$ implies $(ii)$.

Concerning lower semi-continuous retractions, one can establish the
following analogue of Theorem 4.1.

\begin{theorem}\label{lsretract}
Let $X$ be a subspace of $Y$. Then the following conditions are
equivalent:
\begin{itemize}
\item[(i)] There exists a lower semi-continuous compact-valued map\\ $r\colon Y\to\beta X$ with
$r(x)=\{x\}$ for all $x\in X$;
\item[(ii)] There exists an extender
$u:C^*(X)\rightarrow C^*_{usc}(Y)$ which is normed, weakly additive,
preserves $\min$ and weakly preserves $\max$;
\item[(iii)] There exists an extender
$u:C^*(X)\rightarrow C^*_{lsc}(Y)$ which is normed, weakly additive,
preserves $\max$ and weakly preserves $\min$.
\end{itemize}
Moreover, the extenders from $($ii$)$ and $($iii$)$ have compact
supports iff $r(y)\subset X$ for all $y\in Y$.
\end{theorem}

%%%%%%%%%%%%%%%%%%%%%%%%%%%%%%%%%%%%%%%%%%%%%%%%%%%%%%%%%%%%%%%%%%%%%%
\section{Concluding remarks}

Considering extenders which preserve both $\max$ and $\min$, we have
the following proposition:

\begin{proposition}
Let $X$ be a subspace of $Y$. Then each of the following two
conditions implies the existence of a neighborhood $G$ of $X$ in $Y$
and an upper semi-continuous map $r\colon G\to\beta X$ with compact
connected values such that $r(x)=\{x\}$ for all $x\in X$:
\begin{itemize}
\item[(i)] There exists an extender
$u:C^*(X)\rightarrow C^*_{lsc}(Y)$ with $u(1_X)=1_Y$ such that $u$
preserves both $\max$ and $\min$;
\item[(ii)] There exists an extender
$u:C^*(X)\rightarrow C^*_{usc}(Y)$ with $u(1_X)=1_Y$ such that $u$
preserves both $\max$ and $\min$.
\end{itemize}
\end{proposition}

\begin{proof}
Suppose  $u:C^*(X)\rightarrow C^*_{usc}(Y)$ is an extender
satisfying condition $(ii)$. We define the operator
$\mathrm{e}\colon\mathcal T_X\to\mathcal T_Y$ as in the proof of
Theorem 4.3, implication $(iii)\Rightarrow (i)$. Let
$G=\bigcup\{\mathrm{e}(U):U\in\mathcal T_X\}$ and
$r(y)=\bigcap\{\overline{U}^{\beta X}:y\in\mathrm{e}(U)\}$ for all
$y\in G$. We need to show that the values of $r$ are connected.

Suppose $r(y_0)$ is not connected for some $y_0\in G$. So, there are
two non-empty open sets $U_1, U_2$ in $\beta X$ with disjoint
closures such that $r(y_0)\cap U_j\neq\varnothing$, $j=1,2$, and
$r(y_0)\subset U_1\cup U_2$. Fix finitely many open sets
$W_i\subset\beta X$, $i=1,..,k$, with
$y_0\in\bigcap_{i=1}^{i=k}\mathrm{e}(W_i\cap X)$ and
$r(y_0)\subset\bigcap_{i=1}^{i=k}\overline{W_i}^{\beta X}\subset
U_1\cup U_2$. Since $\bigcap_{i=1}^{i=k}\mathrm{e}(W_i\cap
X)=\mathrm{e}\big((\bigcap_{i=1}^{i=k}W_i)\cap X\big)$, we can
suppose that $y_0\in\mathrm{e}\big((U_1\cup U_2)\cap X\big)$. Then,
according to the definition of the operator $\mathrm{e}$, there
exists $h_0\in C_{(U_1\cup U_2)\cap X}$ with $y_0\in
u(h_0)^{-1}((-\infty,1))$. Therefore, $h_0(X)\subset (-\infty,1]$
and $h_0(x)=1$ for all $x\in X\backslash (U_1\cup U_2)$. Because
$U_1$ and $U_2$ have disjoint closures, $h_0=\min\{h_1,h_2\}$, where
$h_j(x)=h_0(x)$ if $x\in U_j\cap X$ and $h_j(x)=1$ if $x\not\in
U_j\cap X$, $j=1,2$. Hence,
$u(h_0)(y_0)=\min\{u(h_1)(y_0),u(h_2)(y_0)\}$. So, $y_0\in
u(h_j)^{-1}((-\infty,1))$ for some $j\in\{1,2\}$. Since $h_j\in
C_{U_j\cap X}$, we have $y_0\in\mathrm{e}(U_j\cap X)$. Consequently,
$r(y_0)\subset\overline{U_j}^{\beta X}$ which contradicts the fact
that $r(y_0)$ meets both $\overline{U_1}^{\beta X}$ and
$\overline{U_2}^{\beta X}$. So, the map $r$ has connected values.

Similar arguments work when $u$ satisfies condition $(i)$.
\end{proof}

\begin{cor}
Let $X$ be a compact connected subspace of a space $Y$ and $u$ is an
extender satisfying one of the conditions $(i)$ and $(ii)$ from
Proposition $5.1$. Then there exists an upper semi-continuous map
$r\colon Y\to X$ with compact connected values such that
$r(x)=\{x\}$ for all $x\in X$.
\end{cor}

\begin{proof}
By Proposition 5.1, there exists an upper semi-continuous retraction
$r_1\colon G\to X$ with non-empty compact connected values, where
$G$ is a neighborhood of $X$ in $Y$. Then the map $r\colon Y\to X$,
$r(x)=r_1(x)$ if $x\in G$ and $r(x)=X$ if $x\not\in G$, is the
required retraction.
\end{proof}

According to \cite{dr}, every compactum which is an upper
semi-continuous compact and connected valued retract of a Tychonoff
cube is an $AE(1)$. This result together with Corollary 5.2 yields
the the following one.

\begin{cor}
Let $X$ be a compact connected space such that for any embedding of
$X$ in another space there exists an extender $u$ satisfying one of
the conditions $(i)$ and $(ii)$ from Proposition $5.1$. Then $X\in
AE(1)$.
\end{cor}

The last corollary leads to the following problem:

\begin{question}
Is there any topological description of the class of compacta $X$
such that for every embedding of $X$ in another space $Y$ there
exists an extender satisfying one of the conditions $(i)$ and $(ii)$
from Proposition 5.1.
\end{question}

M. Zarichnyi \cite{z} investigated the functor of idempotent
probability measures. For a compact space $X$ a functional
$\mu\colon C(X)\to\mathbb R$ is called an {\em idempotent measure}
if $\mu$ is normed, weakly additive and preserves $\max$. The space
$I(X)$ of all idempotent probability measures on $X$ is a compact
subspace of $\mathbb R^{C(X)}$. We say that a compactum $X$ is a
{\em Zarichnyi space} if for every embedding of $X$ in another space
$Y$ there exists a normed, weakly additive extender $u\colon C(X)\to
C(Y)$ which preserves $\max$. This is equivalent to the existence of
a map $\theta\colon Y\to I(X)$ such that $\theta(x)=\delta_x$ for
every $x\in X$. In particular, $X$ is a Zarichnyi space provided
$I(X)$ is an absolute retract. According to Corollary 3.5, every
$AE(1)$-compactum is a Zaricnyi space. But there exists a Zaricnyi
space which is not an $AE(1)$. Indeed, let $X$ be a metric infinite
compactum which is not locally connected. Then, by \cite[Theorem
5.3]{brz}, $I(X)$ is homeomorphic to $\mathbb I^{\omega}$.
Consequently, $X$ is a Zarichnyi space. Since $X$ is not locally
connected, $X\not\in AE(1)$. On the other hand, by Corollary 4.4,
any Zarichnyi space is an $AE(0)$.

\begin{question}
Is there any $AE(0)$-space which is not a Zarichnyi space?
\end{question}

Let us note that every compact metric spaces is a Zarichnyi space.
We already observed that for infinite metric compacta. In case $X$
is a finite set of cardinality $n$, then $I(X)$ is homeomorphic to
the $(n-1)$-dimensional simplex (see \cite{z}). Therefore, if there
exists an $AE(0)$-space which is not a Zarichnyi space, it should be
non-metrizable.

\bibliographystyle{amsplain}

\end{document}